\documentclass[reqno]{amsart}

\usepackage{lmodern}

\usepackage{float}
\usepackage[dvipsnames]{xcolor}
\usepackage[colorlinks=true, urlcolor=blue, citecolor=blue, linkcolor=blue, linktocpage=true, hypertexnames=false]{hyperref}
\usepackage{url}
\usepackage{doi}

\usepackage{caption}
\usepackage{ragged2e}
\usepackage{graphicx}
\usepackage{tikz}
\usetikzlibrary{patterns, decorations.markings}

\usepackage{algorithm}
\usepackage{algpseudocode}

\usepackage{enumerate}
\usepackage{enumitem}

\usepackage{amsthm}
\usepackage{thmtools}
\usepackage{chngcntr}

\theoremstyle{plain}
\newtheorem{theorem}{Theorem}[section]
\renewcommand{\thetheorem}{%
	\ifnum\value{section}>0
	\thesection.%
	\else
	\thechapter.%
	\fi
	\arabic{theorem}%
}
\newtheorem{lemma}[theorem]{Lemma}
\newtheorem{proposition}[theorem]{Proposition}
\newtheorem{corollary}[theorem]{Corollary}

\newtheorem{definition}[theorem]{Definition}
\newtheorem{example}[theorem]{Example}
\newtheorem{remark}[theorem]{Remark}

\newtheorem*{remark*}{Remark}
\newtheorem*{question*}{Open question}
\newtheorem*{theorem*}{Theorem}
\newtheorem*{lemma*}{Lemma}
\newtheorem*{proposition*}{Proposition}
\newtheorem*{corollary*}{Corollary}
\newtheorem*{propriety*}{Property}
\newtheorem*{definition*}{Definition}
\newtheorem*{example*}{Example}
\newtheorem*{notation*}{Notation}

\usepackage{amsmath, amsfonts, amssymb}
\usepackage{dsfont}
\usepackage{stmaryrd}
\usepackage{mathabx}

\newcommand{\B}{\mathcal{B}}

\newcommand{\E}{\mathbb{E}}
\newcommand{\F}{\mathcal{F}}
\newcommand{\G}{\mathcal{G}}

\newcommand{\M}{\mathcal{M}}
\newcommand{\N}{\mathbb{N}}

\renewcommand{\P}{\mathbb{P}}

\newcommand{\R}{\mathbb{R}}

\newcommand{\Z}{\mathbb{Z}}

\newcommand{\borel}[1]{\B\left(#1\right)}
\newcommand{\Vol}[1]{\operatorname{Vol}\left(#1\right)}

\newcommand{\card}[1]{\operatorname{card}\left(#1\right)}

\makeatletter
\def\esp{\@ifnextchar[{\@withe}{\@withoute}}
\def\@withe[#1]#2{\E_{#1}\left[#2\right]}
\def\@withoute#1{\E\left[#1\right]}
\def\var{\@ifnextchar[{\@withv}{\@withoutv}}
\def\@withv[#1]#2{\mathbb{V}_{#1}\left(#2\right)}
\def\@withoutv#1{\mathbb{V}\left(#1\right)}
\makeatother

\newcommand{\Id}{\operatorname{Id}}
\newcommand\restr[2]{{
		\left.\kern-\nulldelimiterspace 
		#1
		\littletaller
		\right|_{#2}
}}
\newcommand{\littletaller}{\mathchoice{\vphantom{\big|}}{}{}{}}
\numberwithin{equation}{section}

\title{Palm perturbed measures on Abelian topological groups}
\author{Loïc Thomassey}
\begin{document}

\begin{abstract}
	We show that the perturbation of a Palm measure by an independent process with stationary increments remain a Palm measure.
\end{abstract}

\maketitle


\tableofcontents
	
\section{Introduction}

This paper introduces perturbed Palm measures. Here the focus is on stationary point processes, that is a random configurations of points $\{\xi_{n} : n \in \Z\}$, say on the real line, and whose distribution is invariant under translations. It is a standard fact that, if an existing stationary point process $\xi$ is subjected to an independent stationary displacement field $X$, the resulting point process $\smash{\xi_{X}=\{\xi_{n}+X_{\xi_{n}} : n\in \Z\}}$ remains stationary. \medbreak

In this article, we take a new standpoint. Instead of looking for perturbations preserving the stationarity of a point process $\xi$, we look for perturbations preserving its Palm measure $\smash{\hat{\xi}=\{\hat{\xi}_{n} : n\in \Z\}}$. The latter is loosely speaking the distribution of $\xi$ conditioned on the event $0\in \xi$, and we will show that Palm measures are preserved by stochastic processes with stationary increments. More explicitly, if $B$ is a stochastic process with stationary increments, independent of $\hat{\xi}$, then $\smash{\hat{\xi}_{B}=\{\hat{\xi}_{n}+B_{\hat{\xi}_{n}}: n \in \Z\}}$ remains a Palm measure, which we call a perturbed Palm measure. \medbreak

As any Palm measure is uniquely associated with a stationary point process, this procedure provides a novel way to construct stationary point processes via perturbations of their Palm measure. In what follows, we will generalise this construction to random measures on an Abelian topological group.

\section{Palm measures and weakly stationary stochastic processes }\label{sec:definion}

\subsection{Palm measures and stationarity} 

This section is a short but formal primer on Palm measures, as introduced by Mecke in his seminal paper \cite{Mecke1967}. Most of the proofs are omitted, but the interested reader is invited to consult Daley and Vere-Jones\cite{Daley2003,Daley2008} or Last \cite{Last2009} for further details.\medbreak

Let $(G,+)$ be an additive, Abelian, Hausdorff, locally compact and second-countable topological group endowed with its Borel $\sigma$-algebra $\G$. A measure $\xi$ on $G$ is called locally-finite if $\xi(K)<+\infty$ for all compact sets $K$, and we denote by $\mathrm{M}(G)$, or $\mathrm{M}$, the space of all locally finite positive measures on $G$. This set is equipped with the $\sigma$-algebra $\M(G)$ generated by the mappings
\begin{equation*}
	\Lambda_{K}:\left\{
	\begin{array}{cccc}
		&\mathrm{M}(G)&\longrightarrow&\R_{+}\\
		&\xi&\longmapsto&\xi(K)
	\end{array}
	\right., \quad K \text{ compact}.
\end{equation*}

We now provide two fundamental examples of locally finite measures.\medbreak

\subsubsection{Haar measures} Say a measure $\lambda_{G}$ is \textit{invariant} on $G$ if
\begin{equation*}
	\int_{G} f(s+t)\, \lambda_{G}(\mathrm{d}s) = \int_{G} f(s)\, \lambda_{G}(\mathrm{d}s), \quad t\in G,
\end{equation*}
for all measurable mappings $f : G \longrightarrow \R_{+}$. A landmark result in the field of topological groups, see Folland \cite[Chapter 2]{Folland2016}, is the existence and uniqueness, up to a multiplicative factor, of an invariant and locally finite measure $\lambda_{G}$. The latter is known as the \textbf{Haar measure} of $G$.\medbreak

\subsubsection{Counting measures} The second instance of locally finite measures we shall encounter in this paper are \textbf{counting measures}. Those consist of locally finite Borel measures of the form
\begin{equation*}
	\xi=\sum_{n\in \Z}\mu_{n}\delta_{\xi_{n}},
\end{equation*}
where $\mu_{n}> 0$ and $\{\xi_{n} : n\in \Z\}$ is a locally finite set. $\xi_{n}$ are known as the atoms of $\xi$ and if each atom has multiplicity $\mu_{n}=1$, $\xi$ is said to be a \textbf{simple}. In what follows, the sets of counting measures and simple counting measures are referred respectively as $\mathrm{N}(G)$, or $\mathrm{N}$ and $\mathrm{N}_{s}(G)$, or $\mathrm{N}_{s}$.

\begin{notation*}
	There is a one-to-one correspondence between simple counting measures $\xi$ and locally finite closed sets $F$. When referring to a simple counting measure,  we will therefore write $\xi=F$ instead of $\xi=\sum_{t\in F}\delta_{t}$.
\end{notation*}

Finally, we introduce some definitions. We say that a random variable $\xi$ with distribution $\mathrm{P}$ is a \textbf{random measure}, and 
\begin{itemize}
	\item a \textbf{(non-simple) point-process} if it is supported on $\mathrm{N}(G)$,
	\item a \textbf{(simple) point process} if it supported on $\mathrm{N}_{s}(G)$.
\end{itemize}

\subsubsection{Stationarity and Palm measures}\label{sec:stationarity_palm}

In this section, $\mathrm{P}$ refers to a $\sigma$-finite Borel measure on $\mathrm{M}(G)$, not necessarily finite. \medbreak

The translations 
\begin{equation*}
	\tau_{t} : \left\{
	\begin{array}{cccc}
		&\mathrm{M}(G) &\longrightarrow& \mathrm{M}(G)\\
		&\mu&\longmapsto&\xi(.+t)
	\end{array}
	\right., \quad t\in G,
\end{equation*}
induce a continuous group action of $G$ on $\mathrm{M}(G)$, and an event $A \in \mathcal{M}(G)$ is called $\mathbf{\tau}$\textbf{-invariant} if 
\begin{equation*}
	\tau_{t}A=A, \quad t\in G.
\end{equation*}
Following these definitions, we say that $\mathrm{P}$ is
\begin{itemize}
	\item  \textbf{stationary} if its distribution is invariant with respect to translations, 
	\begin{equation*}
		\tau_{t} \circ \mathrm{P}=\mathrm{P}, \quad t\in G,
	\end{equation*}
	\item \textbf{ergodic} if it is stationary, and if either $\mathrm{P}(A)=0$ or $\mathrm{P}({}^{c}A)=0$ for any $\tau$-invariant event $A \in \mathcal{M}(G)$.
\end{itemize}

As the distribution of $\mathrm{P}$ is invariant with respect to translations, one can informally imagine re-centring the process around one of its atom. This idea is fruitful and motivates the introduction of the \textbf{Palm measure} of $\mathrm{P}$. The latter, denoted by $\mathrm{P}_{0}$, is the $\sigma$-finite measure on $\mathrm{M}(G)$ defined by the change of measure
\begin{equation}\label{eq:Palm_measure}
	\mathrm{P}_{0}(A)=\int_{\mathrm{M}}\int_{G} \omega(t)\mathds{1}(\tau_{t}\xi \in A) \, \xi(\mathrm{d}t)\,\mathrm{P}_{0}(\mathrm{d}\xi)
\end{equation}	
where $\omega : G \longrightarrow \R_{+}$ is any measurable mapping satisfying 
\begin{equation*}
	\int_{G}\omega(t)\, \mathrm{d}t =1,
\end{equation*}
the preceding definition being independent of the choice of $\omega$. $\mathrm{P}_{0}$ is itself a $\sigma$-finite measure of mass
\begin{equation*}
	\lambda_{P}=\int_{\mathrm{M}}\int_{G}\omega(t)\, \xi(\mathrm{d}t)\, \mathrm{P}(\mathrm{d}\xi).
\end{equation*}
where $\lambda_{P}$ is known as the \textbf{intensity} of $\mathrm{P}$. When $\lambda_{P}$ is finite, $\mathrm{P}_{0}$ can be renomalized into a probability measure $\P_{0}$, known as the \textbf{Palm distribution} of $\mathrm{P}$. \medbreak

So far, we have seen that any stationary measure $\mathrm{P}$ is associated with a unique Palm measure $\mathrm{P}_{0}$. It is natural to wonder whether the converse holds true, that is whether a Palm measure identifies uniquely with a stationary measure. Historically, this has been the main motivation behind the development of Palm measures and this question is intrinsically related to Campbell's refined theorem \cite[Theorem 2.6]{Last2009}. We won't state Campbell's theorem, but rather recall a corollary known as \textbf{Palm inversion formula}, see \cite[Theorem 13.2.V]{Daley2008}. The latter states, that up an atomic part, $\mathrm{P}$ can be explicitly recovered via its Palm measure $\mathrm{P}_{0}$ as
\begin{equation}\label{eq:inverse_Palm_measure}
	\mathrm{P}(A)=\int_{\mathrm{M}}\int_{G}\, \mathds{1}(\tau_{-t}\xi \in A)k(\tau_{-t}\xi, t)\, \lambda_{G}(\mathrm{d}t)\, \mathrm{P}_{0}(\mathrm{d}\xi) + \mathrm{P}(\{\emptyset\})\mathds{1}(\emptyset \in A)
\end{equation}
with $\emptyset$ denoting the zero-measure on $G$ and $k : \mathrm{M}\times G\longrightarrow \R_{+}$ a kernel satisfying
\begin{equation}\label{eq:kernel}
	\int_{G} k(\xi, t)\, \xi(\mathrm{d}t) = 1, \quad \xi \neq \emptyset.
\end{equation}

The preceding equation has important implications. First, it does not depend on the choice of $k$. Second, it underlines the fact that $\mathrm{P}$ can be entirely recovered by the knowledge of $\mathrm{P}_{0}$, up to an atomic part supported on $\{\emptyset\}$. Consequently, there is a one-to-one correspondence between stationary measures $\mathrm{P}$ satisfying $\mathrm{P}(\{\emptyset\})=0$ and Palm measures $\mathrm{P}_{0}$. To avoid any unnecessary difficulty, we will from now on assume that $\mathrm{P}(\{\emptyset\})=0$ so that any Palm measure is uniquely associated with a stationary measure.

\begin{example}
	One of the most striking application of Campbell's inversion formula concerns simple point processes on the real line, see \cite[Theorem 13.3.I]{Daley2008}. Assume $\mathrm{P}$ is supported on $\mathrm{N}_{s}(\R)$ and write $\xi=\{\xi_{n}:n\in \Z\}$ with the conventions
	\begin{itemize}
		\item $\xi_{-1}<0\leq\xi_{0}$,
		\item $\xi_{n}<\xi_{n+1}$ for all $n\in \Z$,
	\end{itemize}
	for $\xi \in \mathrm{N}_{s}(\R)$. Then 
	\begin{equation}\label{eq:inverse_palm_real_line}
		\mathrm{P}(A)=\int_{\mathrm{N}_{s}}\int_{0}^{\xi_{1}}\mathds{1}(\tau_{-t}\xi \in A) \, \mathrm{d}t\, \mathrm{P}_{0}(\mathrm{d}\xi).
	\end{equation}
\end{example}

Finally, one of the main advancements in the field of Palm theory is a characterisation of Palm measures due to Mecke \cite{Mecke1967}.

\begin{theorem}[{\cite[Theorem 7.1]{Last2009}, \cite[Theorem 13.2.VIII]{Daley2008}}]\label{thm:palm_characterization}
	A measure $\mathrm{P}_{0}$ on $\mathrm{M}(G)$ is the Palm measure of a stationary measure $\mathrm{P}$ on $\mathrm{M}(G)$ if and only if
	\begin{enumerate}
		\item \label{enum:palm_}$\mathrm{P}_{0}$ is $\sigma$-finite,
		\item $\mathrm{P}_{0}(\{\emptyset\})=0$ where $\emptyset$ denotes the zero measure,
		\item the measure $\mathrm{Q}(\mathrm{d}\xi,\mathrm{d}t)=\mathrm{P}_{0}(\mathrm{d}\xi)\, \xi(\mathrm{d}t)$ is invariant with respect to 
		\begin{equation*}
			\rho : \left\{
			\begin{array}{cccc}
				&\mathrm{M}(G)\times G&\longrightarrow &\mathrm{M}(G)\times G\\
				&(\xi,t)&\longmapsto &(\tau_{t}\xi,-t)
			\end{array}
			\right..
		\end{equation*}
	\end{enumerate}
\end{theorem}

For practical applications, the third condition can be restated in an integral form, namely 
\begin{equation*}
	\int_{\mathrm{M}}\int_{G}f(\tau_{t}\xi,-t) \, \xi(\mathrm{d}t)\, \mathrm{P}_{0}(\mathrm{d}\xi)=\int_{\mathrm{M}}\int_{G}f(\xi,t)\, \xi(\mathrm{d}t)\, \mathrm{P}_{0}(\mathrm{d}\xi)
\end{equation*}
for any positive measurable function $f : \mathrm{M}(G)\times G\longrightarrow \R$. Informally, this condition means that the stationarity of a measure $\mathrm{P}$ is equivalent to the invariance of its Palm measure $\mathrm{P}_{0}$ with respect to $\rho$. This statement is extremely surprising as it relates stationarity of a continuous dynamical system (translations) to the stationarity of a discrete dynamical system ($\rho$). This seemingly innocuous observation will be relevant when dealing with ergodicity of perturbed Palm measures in the subsequent section.\medbreak

We will finally conclude this introduction on Palm measures with an important and relevant observation, namely Palm measures preserve counting measures. In fact,
\begin{itemize}
	\item $\mathrm{P}$ is supported on $\mathrm{N}(G)$ if and only if  $\mathrm{P}_{0}$ is supported on $\mathrm{N}(G)$,
	\item $\mathrm{P}$ is supported on $\mathrm{N}_{s}(G)$ if and only if $\mathrm{P}_{0}$ is supported on $\mathrm{N}_{s}(G)$.
\end{itemize}

\subsection{Stochastic processes with stationary increments}\label{sec:stationary_increments}

Informally, the reader should view a stochastic process with stationary increments as a generalisation of a real Brownian motion $(B_{t})_{t\in \R}$, except it is defined in a topological group setting.\medbreak

Let $\mathrm{F}(G)$, or simply $\mathrm{F}$ be the space of (continuous) functions $B : G \longrightarrow G$ satisfying 
\begin{enumerate}[label=(\text{C\arabic*})]
	\item \label{cond:C1}$B_{0}=0$,
	\item \label{cond:C2} $(\Id+B)\circ \xi$ is a finite locally measure for any $\xi \in \mathrm{M}(G)$, that is
	\begin{equation*}
		\int_{G} \mathds{1}_{K}(t+B_{t})\, \xi(\mathrm{d}t)<+\infty, \quad \text{$K \subset G$ compact}.
	\end{equation*}
\end{enumerate}

Here, $(\Id +B)\circ \xi$ denotes the push-forward measure of $\xi$ by $\Id+B$.\medbreak

\begin{remark}
	When defining a process with stationary increments, condition \ref{cond:C2} is irrelevant. Yet, it will prove useful later on when introducing perturbed Palm measures as it will ensure the latter are well-defined. For now, it can be omitted.
\end{remark}

We equip $\mathrm{F}(G)$ with its cylindrical algebra $\mathcal{F}(G)$ and say that a $\sigma$-finite measure $\nu$ on $\mathrm{F}(G)$ is 
\begin{itemize}
	\item \textbf{weakly-stationary} if its distribution has \textbf{stationary increments} in the sense that 
	\begin{equation*}
		\theta_{t} \circ \nu = \nu, \quad t\in G,
	\end{equation*}
	where
	\begin{equation*}
		\theta_{t} : \left\{
		\begin{array}{cccc}
			&\mathrm{F}(G)&\longrightarrow&\mathrm{F}(G)\\
			&B&\longmapsto&\quad (B_{t+s}-B_{t})_{s\in G}
		\end{array}
		\right., \quad t\in G.
	\end{equation*}
	\item \textbf{ergodic} if $\nu$ is weakly-stationary, and if either $\nu(A)=0$ or $\nu({}^{c}A)=0$ for any $\mathbf{\theta}$\textbf{-invariant} event $A\in \mathcal{F}(G)$ in the sense
	\begin{equation*}
		\theta_{t}A=A, \quad t\in G.
	\end{equation*}
\end{itemize}
As before, $\nu$ might not be necessarily a probability measure.

\begin{notation*}
	If $B$ is a random variable with weakly-stationary distribution $\nu$, we say that $B$ has stationary increments.
\end{notation*}

\begin{example}
	Let $\nu = \mathcal{L}(B)$ be the distribution of a Brownian motion. As a Brownian motion has continuous trajectories and stationary increments, \textit{i.e.} the distribution of $(B_{s+t}-B_{s})_{t\in \R}$ does not depend on $t$, $\nu$ is weakly-stationary. 
\end{example}

Before concluding this section, we make a few remarks on the definition. First, we assume $B$ to be continuous even though there are numerous known examples of stochastic processes with stationary increments that have non-continuous sample paths, say a Levy process with jumps. This is a slight inconvenience, but the reader should be aware that continuity is mostly a technical assumption, and we firmly believe most results of the upcoming sections still hold true if the latter is relaxed.\medbreak

Finally, when it comes to practical applications, condition \ref{cond:C2} is not really tractable. In $\R^{d}$, the latter can be replaced with the (weaker) assumption
\begin{enumerate}[resume,label=(\text{C\arabic*})]
	\item \label{cond:sub-linear}$B$ has sub-linear growth, \textit{i.e.}
	\begin{equation*}
		\lim_{\|t\|\to +\infty}\frac{\|B_{t}\|}{\|t\|} < 1.
	\end{equation*}
\end{enumerate}

We will henceforth denote by $\mathrm{F}_{0}(\R)$ the sets of functions $B:G\longrightarrow G$ satisfying conditions \ref{cond:C1} and \ref{cond:sub-linear}. We let the reader check by himself that $\mathrm{F}_{0}(\R^{d})$ is a subset of $\mathrm{F}(\R^{d})$.

\section{Perturbation of Palm measures}\label{sec:perturbation_palm}

In this section, we introduce the main objects of this article, namely perturbed Palm measures. Those were briefly introduced in the introduction for point processes on the real line. Here we extend the definition to random measures in an Abelian topological group.

\subsection{Perturbed Palm measures}

\subsubsection{Definition and main theorem}

Let 
\begin{equation}\label{eq:Gamma}
	\Gamma : \left\{
	\begin{array}{cccc}
		&\mathrm{M}(G)\times \mathrm{F}(G)&\longrightarrow& \mathrm{M}(G)\\
		&(\xi,B)&\longmapsto&(\Id+B)\circ \xi
	\end{array}
	\right..
\end{equation}

More explicitly, $\Gamma(\xi,B)$ is the locally finite Borel measure $\xi_{B}$ defined by 
\begin{equation*}
	\xi_{B}(\varphi) = \int \varphi(t+B_{t}) \,\xi(\mathrm{d}t)
\end{equation*}
and we will often prefer the notation $\xi_{B}$ to $\Gamma(\xi,B)$ as it is less cumbersome.\medbreak

If
\begin{itemize}
	\item $\mathrm{P}_{0}$ is a measure on $\mathrm{M}(G)$,
	\item $\nu$ is a measure on $\mathrm{F}(G)$,
\end{itemize}	
we denote by $\mathrm{P}_{0}\otimes \nu$ the product measure of $\mathrm{P}_{0}$ and $\nu$, and
\begin{equation*}
	\mathrm{P}_{0}^{\nu} = \Gamma \circ (\mathrm{P}_{0}\otimes \nu).
\end{equation*}
the push-forward of $\mathrm{P}_{0}\otimes \nu$ by $\Gamma$. \medbreak

We will now state the main result of this article, namely
\begin{theorem}[Perturbed Palm measures]\label{thm:Palm_perturbation}
	Let 
	\begin{itemize}
		\item $\mathrm{P}_{0}$ be the Palm measure of a stationary measure $\mathrm{P}$ on $\mathrm{M}(G)$,
		\item $\nu$ a weak-stationary measure on $\mathrm{F}(G)$,
	\end{itemize}
	
	If $\mathrm{P}_{0}^{\nu}$ is $\sigma$-finite, then there exists a $\sigma$-finite stationary measure $\mathrm{P}^{\nu}$ on $\mathrm{M}(G)$ so that $\mathrm{P}_{0}^{\nu}$ is the Palm measure of $\mathrm{P}^{\nu}$.
\end{theorem}

In the introduction, we teased that Palm measures were preserved by stochastic processes with stationary increments. Theorem \ref{thm:Palm_perturbation} is a reformulation of this idea in a more abstract setting. \medbreak

We now extend the conclusion of Theorem \ref{thm:Palm_perturbation} to ergodic stationary measures. Say 

\begin{itemize}
	\item $A\in \mathcal{M}(G)\otimes\mathcal{F}(G)$ is $\mathbf{(\tau, \theta)}$\textbf{-invariant} if 
	\begin{equation*}
		(\tau_{t},\theta_{t})A = A, \quad t\in G,
	\end{equation*}
	\item$(\mathrm{P},\nu)$ are  \textbf{jointly-ergodic} if either $\mathrm{P}\otimes\nu(A)=0$ or $\mathrm{P}\otimes\nu({}^{c}A)=0$ for any $(\tau, \theta)$-invariant event $A$.
\end{itemize}

\begin{theorem}\label{thm:Palm_perturabtion_ergodicity}
	Under the assumptions of Theorem \ref{thm:Palm_perturbation}, $\mathrm{P}^{\nu}$ is ergodic if $(\mathrm{P},\nu)$ are jointly-ergodic.
\end{theorem}

As the formulation of the preceding results might not be too enlightening, we will proceed to reformulate them in a probabilistic framework. Let $\xi$, $\smash{\hat{\xi}}$ and $B$ be mutually independent random variables such that
\begin{itemize}
	\item $\xi$ is stationary with distribution $\mathrm{P}$,
	\item $\hat{\xi}$ has distribution $\mathrm{P}_{0}$,
	\item $B$ is a stochastic process with stationary increments and distribution $\nu$,  
\end{itemize}
then $\hat{\xi}_{B}=\Gamma(\hat{\xi},B)$ is the Palm distribution of a stationary random measure on $G$. Moreover, if $\xi$ and $B$ are jointly ergodic, then so is the perturbed Palm measure.

\begin{remark}
	We finish this section with some technical remarks related to the definition of $\Gamma$.\medbreak
	
	First, for $\Gamma$ to be well-defined, we used implicitly condition \ref{cond:C2}. If we remove the latter, $\xi_{B}$ would still be a Borel measure, but it would not be necessarily locally finite, see Example \ref{ex:counter_example_locally_finite}.\medbreak
	
	Second, the well-definiteness of $\Gamma$ is not sufficient to define properly $\mathrm{P}_{0}^{\nu}$. Instead, we need $\Gamma$ to be measurable. We shall not delve extensively into this question and refer the reader to Appendix \ref{app:compact-open} for a proof of this fact.
\end{remark}

\begin{example}\label{ex:counter_example_locally_finite}
	Let $\mathrm{d}x$ denote the Lebesgue measure on $\R^{d}$, and define $B_{x}=-x$. Then a straightforward computation shows that $\Gamma(\mathrm{d}x,B)$ is a Dirac measure with infinite mass concentrated at the origin. Hence, it is not locally finite. This example motivates the introduction of condition \ref{cond:C2}.
\end{example}

\subsubsection{Counting measures}

Theorem \ref{thm:Palm_perturbation} can be specialised to counting measures, observing that, if
\begin{equation*}
	\xi = \sum_{n\in \Z}\mu_{n}\delta_{\xi_{n}}
\end{equation*} 
is a counting measure, then
\begin{equation*}
	\xi_{B}=\sum_{n\in \Z}\mu_{n}\delta_{\xi_{n}+B_{\xi_{n}}}.
\end{equation*} 
is also a counting measure. Therefore,

\begin{corollary}\label{cor:palm_perturbation_counting_measure}
	Under the assumptions of Theorem \ref{thm:Palm_perturbation}, if $\mathrm{P}$ is supported on $\mathrm{N}(G)$, then so is $\mathrm{P}^{\nu}$.
\end{corollary}

The preceding corollary is no longer true for simple counting measures as $\xi_{B}$ does not remain necessarily simple. As a counter-example, consider $G$ a finite group, $\xi = G$ and $B_{x}=-x$. A straightforward computation shows that $\xi_{B} =\card{G} \delta_{0}$, so that $\xi_{B}$ is never simple if $G$ has more than $1$ element.\medbreak

In what follows, we will investigate a sufficient condition ensuring $\mathrm{P}^{\nu}$ to be supported on $\mathrm{N}_{s}(G)$. This will strengthen the conclusion of Corollary \ref{cor:atomic_measure} to simple counting measures.\medbreak

Say a weakly-stationary measure $\nu$ is \textbf{pointwise non-atomic} if $\nu_{t}=\pi_{t}\circ \nu$ has no atom for all $t\neq 0$ where
\begin{equation*}
	\pi_{t} : \left\{
	\begin{array}{cccc}
		&\mathrm{F}&\longrightarrow&G\\
		&B&\longmapsto&B_{t}
	\end{array}
	\right.,
\end{equation*} 
In probabilistic terms, if $B$ has distribution $\nu$, then $B$ is pointwise non-atomic if for all $t\neq 0$,
\begin{equation*}
	\nu(B_{t}=x)=0, \quad x\in G.
\end{equation*}
Note that for $t=0$, $B_{0}=0$ $\nu$-almost surely, hence $B_{0}$ is always atomic.
\begin{corollary}\label{cor:atomic_measure}
	Let $\mathrm{P}_{0}$ and $\nu$ as in Theorem \ref{thm:Palm_perturbation} with the additional conditions that
	\begin{enumerate}
		\item $\mathrm{P}_{0}$ is supported on $\mathrm{N}_{s}(G)$,
		\item $\nu$ is pointwise non-atomic.
	\end{enumerate}
	Then $\mathrm{P}^{\nu}$ is supported on $\mathrm{N}_{s}(G)$.
\end{corollary}

We conclude this paragraph with an example.

\begin{example}
	Let $B$ be a Brownian motion. $B$ is pointwise non-atomic as $B_{t}$ is absolutely continuous with respect to Lebesgue measure for any $t\neq 0$. As a consequence, the process
	\begin{equation*}
		\Z_{B}=\{n+B_{n} : n \in \Z\}
	\end{equation*}
	is the Palm measure of a stationary simple point process.
\end{example}

\subsubsection{A probabilistic interpretation} \label{sec:probabilistic_interpretration}

So far, we introduced perturbed Palm measures in a measure-theoretical framework and we deliberately avoided any reference to probability spaces. There is a legitimate reason as perturbed Palm measures do not preserve probabilities. More precisely, $\mathrm{P}^{\nu}$ has no reason to be a probability measure even if both $\mathrm{P}$ and $\nu$ is are. Worse, as underlined by Proposition \ref{prop:counter_example_probability}, $\mathrm{P}^{\nu}$ can be infinite even if $\mathrm{P}$ and $\nu$ are both finite.\medbreak

In the rest of this section, the goal is to derive conditions on both $\mathrm{P}$ and $\nu$ so that the preceding situation does not occur. In fact, we simply need $\mathrm{P}^{\nu}$ to be finite so that it can be renormalised into a probability measure. \medbreak 

In virtue of Campbell's inversion formula \eqref{eq:inverse_Palm_measure}, $\mathrm{P}^{\nu}$ is finite  if and only if there exists a kernel $k$ satisfying \eqref{eq:kernel} and so that
\begin{equation}\label{eq:finiteness_perturbed_palm_measure}
	\int_{G}k(\tau_{-t}\xi, t) \, \lambda_{G}(\mathrm{d}t)\, \mathrm{P}_{0}(\mathrm{d}\xi) <+\infty.
\end{equation}

We consider different scenarios ensuring \eqref{eq:finiteness_perturbed_palm_measure} holds. We don't claim any novelty as those criteria are already well-known in the literature. As before, proofs are deferred to Section \ref{sec:proofs}. \medbreak

We start with the easiest case, $G$ compact. 
\begin{proposition}\label{prop:compact_finite_perturbed_palm_measure}
	If $G$ is compact, then $\mathrm{P}^{\nu}$ is finite if and only if both $\mathrm{P}$ and $\nu$ are finite. Moreover,
	\begin{equation*}
		\mathrm{P}^{\nu}(\mathrm{M})=\mathrm{P}(\mathrm{M})\times \nu(\mathrm{F}).
	\end{equation*}
\end{proposition}

If $G$ is no longer compact, the preceding result no longer holds and $\mathrm{P}^{\nu}$ might be infinite even if $\mathrm{P}$ and $\nu$ are finite. This is the content of the following proposition, which is stated for $G=\Z$, but similar examples can be constructed in $\R^{d}$.

\begin{proposition}\label{prop:counter_example_probability}
	Let $G=\Z$, $\mathrm{P}_{0}=\delta_{\Z}$ be the Palm measure of the stationary lattice $\Z$, and $\nu$ be the distribution of a stochastic process $B$ with independent and identically distributed increments satisfying 
	\begin{itemize}
		\item $(B_{n})_{n\in \Z}$ is increasing.
		\item $\esp{B_{1}}=+\infty$
	\end{itemize}
	Then, $\mathrm{P}^{\nu}$ has infinite mass.
\end{proposition}	

In the rest of this paragraph, we will focus on $\R^{d}$ and simple counting measures. We will therefore assume that both $\mathrm{P}$ and $\mathrm{P}^{\nu}$ are supported on $\mathrm{N}_{s}(\R^{d})$, or equivalently, that both $\mathrm{P}_{0}$ and $\mathrm{P}_{0}^{\nu}$ are supported on $\mathrm{N}_{s}(\R^{d})$.\medbreak

The \textbf{Voronoi diagram} of a configuration of points $\xi$ is a partition of $\R^{d}$ consisting of convex polygonal cells
\begin{equation*}
	V_{x}(\xi)=\left\{t\in \R^{d} : |t-x|\leq \inf_{y \in \xi} |t-y|\right\}.
\end{equation*}
Alternatively, $t\in V_{x}(\xi)$ if $t$ is closer to $x$ than any other $y \in \xi$.\medbreak

Note that $\mathrm{P}_{0}(0\not\in \xi)=0$ so that $0$ is $\mathrm{P}_{0}$-almost surely an atom of $\xi$. Now the volume of $0$-Voronoi cell of $\xi$ is intrinsically related to the mass of $\mathrm{P}$ via the equation \cite[Equation 13.3.9b]{Daley2008}
\begin{equation}\label{eq:voronoi_cell}
	\mathrm{P}(\mathrm{N}_{s})=\int_{\mathrm{N}_{s}}\Vol{V_{0}(\xi)}\, \mathrm{P}_{0}(\mathrm{d}\xi).
\end{equation}

Consequently, $\mathrm{P}$ is finite if and only if the expected volume of the $0$-Voronoi is finite. For Palm perturbed measures, this translates into

\begin{proposition}\label{prop:Voronoi_cell}
	If $\mathrm{P}_{0}^{\nu}$ is supported on $\mathrm{N}_{s}(\R^{d})$, then $\mathrm{P}^{\nu}$ is finite if and only if 
	\begin{equation*}
		\int_{\mathrm{N}_{s}}\int_{\mathrm{F}}\int_{\R^{d}} \Vol{V_{0}(\xi_{B})} \mathrm{P}_{0}(\mathrm{d}\xi)\,\nu(\mathrm{d}B)<+\infty.
	\end{equation*}
\end{proposition}

Proposition \ref{prop:Voronoi_cell} provides an explicit criteria for perturbed Palm measures in $\R^{d}$ as long as there is a possibility of upper bounding the expected volume of the $0$-Voronoi cell of $\xi_{B}$. The latter might be hard in practice.

\subsection{Proofs}\label{sec:proofs}

\subsubsection{Proof of stationarity}

We now turn to the proof of Theorem \ref{thm:Palm_perturbation}. The latter is based on two fundamental and simple ideas which are summarised in the following lemma.

\begin{lemma}\label{lem:technical_lemma_Palm_perturbation}
	Let $\xi \in \mathrm{M}(G)$, $f\in \mathrm{F}(G)$ and $t\in G$. Then,
	\begin{equation}\label{eq:invariance}
		B_{t}=-\theta_{t}B_{-t}
	\end{equation}
	and 
	\begin{equation}\label{eq:compatibility_equation}
		\tau_{t+B_{t}} \xi_{B}=\Gamma(\tau_{t}\xi, \theta_{t}B)
	\end{equation}
\end{lemma}

\begin{proof}[Proof of Lemma \ref{lem:technical_lemma_Palm_perturbation}]
	With $B_{0}=0$, we have
	\begin{equation*}
		B_{t}=-(B_{-t+t}-B_{t})=-\theta_{t}B_{-t}.
	\end{equation*}
	For the second equation, it is derived from a short mathematical computation. Let $\varphi:G\longrightarrow\R_{+}$ be measurable, then
	\begin{align*}
		\tau_{t+B_{t}} \xi_{B}(\varphi)&=\int_{G}\varphi(s-t-B_{t})\, \xi_{B}(\mathrm{d}s)\\
		&= \int_{G}\varphi(s-t+B_{s}-B_{t}) \, \xi(\mathrm{d}s)\\
		&=\int_{G} \varphi(s-t+\theta_{t}B_{s-t})\, \xi(\mathrm{d}s)\\
		&=\int_{G} \varphi(s+\theta_{t}B_{s})) \, \tau_{t}\xi(\mathrm{d}s)\\
		&=\Gamma(\tau_{t} \xi,\theta_{t} f)(\varphi).
	\end{align*} 
	This finishes the proof.
\end{proof}

We are now ready to prove Theorem \ref{thm:Palm_perturbation}.

\begin{proof}[Proof of Theorem \ref{thm:Palm_perturbation}]
	It suffices to check that the conditions of Mecke's characterisation of Palm measures, see Theorem \ref{thm:palm_characterization}, are satisfied. The first two conditions are easily checked as $\mathrm{P}_{0}^{\nu}$ is assumed $\sigma$-finite and $\mathrm{P}_{0}^{\nu}(\{\emptyset\})=\mathrm{P}_{0}(\{\emptyset\})\nu(\mathrm{F}(G))=0$.\medbreak 
	
	For the third one, we introduce the measures
	\begin{equation*}
		\left\{\
		\begin{array}{ccc}
			\mathrm{Q}(\mathrm{d}\xi,\mathrm{d}t) &=&  \mathrm{P}_{0}(\mathrm{d}\xi)\,\xi(\mathrm{d}t)\\
			\mathrm{Q}^{\nu}(\mathrm{d}\xi,\mathrm{d}t)&=& \mathrm{P}_{0}^{\nu}(\mathrm{d}\xi)\, \xi(\mathrm{d}t)
		\end{array}
		\right.
	\end{equation*}
	By definition of $\mathrm{P}_{0}^{\nu}$, one has
	\begin{align}\label{eq:explicit_formula}
		\mathrm{M}^{\nu}(\varphi)=\int_{\mathrm{M}\times G}\int_{\mathrm{F}}\varphi(\xi_{B},t+B_{t})\, \nu(\mathrm{d}B)\, \mathrm{Q}(\mathrm{d}\xi, \mathrm{d}t).
	\end{align}
	
	In virtue of Mecke's theorem, we need to prove that $\mathrm{M}^{\nu}$ is $\rho$-invariant where 
	\begin{equation*}
		\rho(\xi,t)=(\tau_{t}\xi,-t).
	\end{equation*}
	Consider $\varphi : \mathrm{M}\times G \longrightarrow \R_{+}$ a measurable mapping, then
	\begin{equation*}
		\begin{aligned}
			\rho \circ \mathrm{M}^{\nu}(\varphi)=&\int_{\mathrm{M}\times G} \varphi(\tau_{t}\xi, -t)\, \mathrm{Q}^{\nu}(\mathrm{d}t, \mathrm{d}\xi)\\
			\text{\eqref{eq:explicit_formula}}=&\int_{\mathrm{M}\times G}\int_{\mathrm{F}} \varphi\left(\tau_{t+B_{t}}\xi_{B}, -t-B_{t}\right)\, \nu(\mathrm{d}B)\, \mathrm{Q}(\mathrm{d}t, \mathrm{d}\xi)\\
			\text{(Lemma \ref{lem:technical_lemma_Palm_perturbation})}=&\int_{\mathrm{M}\times G}\int_{\mathrm{F}} \varphi\left(\Gamma(\tau_{t}\xi,\theta_{t}B), -t+\theta_{t}B_{-t}\right)\, \nu(\mathrm{d}B)\, \mathrm{Q}(\mathrm{d}t, \mathrm{d}\xi)\\
			\text{($\nu$ weakly-stationary)}=&\int_{\mathrm{M}\times G}\int_{\mathrm{F}}f\left(\Gamma(\tau_{t}\xi,B), -t+B_{-t}\right)\, \nu(\mathrm{d}f)\, \mathrm{Q}(\mathrm{d}t, \mathrm{d}\xi)\\
			=&\int_{\mathrm{M}\times G}\int_{\mathrm{F}}\psi_{B}(\tau_{t}\xi, -t)\, \nu(\mathrm{d}B)\, \mathrm{Q}(\mathrm{d}\xi,\mathrm{d}t)
		\end{aligned}
	\end{equation*}
	where 
	\begin{equation*}
		\psi_{B}(\xi,t) = \varphi\left(\xi_{B}, t+B_{t}\right)
	\end{equation*}
	As $\mathrm{P}_{0}$ is a Palm measure, Theorem \ref{thm:palm_characterization} ensures that $\mathrm{M}$ is invariant with respect to $\rho$. Consequently,
	\begin{equation*}
		\begin{aligned}
			\rho \circ \mathrm{M}^{\nu}(\varphi) =&\int_{\mathrm{M}\times G}\int_{\mathrm{F}}\psi_{B}(\xi, t)\, \mathrm{Q}(\mathrm{d}\xi,\mathrm{d}t)\, \nu(\mathrm{d}B)\\
			=&\int_{\mathrm{M}\times G}\int_{\mathrm{F}}\varphi\left(\xi_{B}, t+B_{t}\right)\, \nu(\mathrm{d}B)\, \mathrm{Q}(\mathrm{d}t, \mathrm{d}\xi)\\
			=&\int_{\mathrm{M}\times G}\varphi\left(\xi, t\right)\, \mathrm{Q}^{\nu}(\mathrm{d}t, \mathrm{d}\xi)\\
			=&\mathrm{M}^{\nu}(\varphi).
		\end{aligned}
	\end{equation*}
	Hence, $\mathrm{M}^{\nu}$ is $\rho$-invariant and the conclusion follows from Mecke's characterisation of Palm measures.
\end{proof}

\subsubsection{Proof of ergodicity}

This paragraph is dedicated to the proof of Theorem \ref{thm:Palm_perturabtion_ergodicity} and the ergodicity of $\mathrm{P}^{\nu}$. The idea is to translate the ergodicity of $\mathrm{P}^{\nu}$ into an explicit condition on its Palm measure $\mathrm{P}_{0}^{\nu}$. This done via extending of Mecke's characterisation of Palm measures. Informally, the latter states that $\mathrm{P}_{0}$ is the Palm measure of a stationary measure if and only if $\mathrm{Q}(\mathrm{d}\xi, \mathrm{d}t)=\xi(\mathrm{d}t)\, \mathrm{P}_{0}(\mathrm{d}\xi)$ is invariant with respect to 
\begin{equation*}
	\rho(\xi,t)=(\tau_{t}\xi,-t).
\end{equation*}

It turns out that, not only stationarity, but also ergodicity is encoded by $\rho$. This is the content of the next proposition:

\begin{proposition}\label{prop:Palm_ergodicity}
	Let $\mathrm{P}$ be a stationary measure on $\mathcal{M}(G)$. The following propositions are equivalent:
	\begin{enumerate}
		\item $\mathrm{P}$ is ergodic.
		\item Either $\mathrm{P}_{0}(A)=0$ or $\mathrm{P}_{0}({}^{c}A)=0$ for any invariant event $A\in \mathcal{M}(G)$.
		\item The discrete dynamical system $(\mathrm{M}(G)\times G,\mathcal{M}(G)\times G,\mathrm{Q}, \rho)$ is ergodic.
	\end{enumerate}
\end{proposition}

We would like to draw the reader's attention to the fact that $\mathrm{M}(G)\times G$ is not equipped the standard product $\sigma$-algebra $\mathcal{M}(G)\otimes\G$, but with the simpler algebra
\begin{equation*}
	\mathcal{M}(G)\times G =\{A \times G : A\in \mathcal{M}(G)\}.
\end{equation*}

In fact, $\mathcal{M}(G)\otimes \G$ is too large as a $\sigma$-field and one can construct a counterexample where $\mathrm{P}$ is ergodic but $\mathrm{Q}$ is not on $\mathcal{M}(G)\otimes \G$, see Example \ref{ex:ergodicity}.

\begin{proof}[Proof of Proposition \ref{prop:Palm_ergodicity}]
	We start by proving that $(1)$ and $(2)$ are equivalent.\medbreak
	
	$(1)\iff (2)$ It suffices to prove that $\mathrm{P}(A)=0$ if and only if $\mathrm{P}_{0}(A)=0$ for any invariant event $A\in \mathcal{M}(G)$.\medbreak
	
	If $\mathrm{P}(A)=0$, Campbell's formula \eqref{eq:Palm_measure} and the $\tau$-invariance of $A$ yields
	\begin{equation*}
		\mathrm{P}_{0}(A)=\int_{A}\int_{G}\omega(t) \, \xi(\mathrm{d}t)\, \mathrm{P}(\mathrm{d}\xi)=0.
	\end{equation*}
	As we are integrating over a set of null measure, it follows that $\mathrm{P}_{0}(A)=0$. This finishes the direct part of the proof.\medbreak
	
	On the other hand, if $\mathrm{P}_{0}(A)=0$, Campbell's inversion formula \eqref{eq:inverse_Palm_measure} coupled with the invariance of $A$ yields
	similarly
	\begin{equation*}
		\mathrm{P}(A)=\int_{A}\int_{G}k(\tau_{t}\xi,-t)\, \lambda_{G}(\mathrm{d}t)\,\mathrm{P}_{0}(\mathrm{d}\xi)=0.
	\end{equation*}
	
	$(2)\iff (3)$
	Note that $A\times G$ is $\rho$-invariant if and only if $A$ is $\tau$-invariant. Moreover, for any $\tau$-invariant event $A$, the definition of $\mathrm{Q}$ yields 
	\begin{align*}
		\mathrm{Q}(A\times G) =\int_{A}\xi(G)\, \mathrm{P}_{0}(\mathrm{d}\xi).
	\end{align*}
	Now, if $\mathrm{P}_{0}(A)=0$, then  $	\mathrm{Q}(A\times G)=0$ and $\mathrm{Q}$ is ergodic.\medbreak
	
	For the converse, note that $\xi(G)=0$ if and only if $\xi = \emptyset$. As a consequence, $\mathrm{Q}(A\times G)=0$ implies that $\mathrm{P}_{0}(A\backslash \{\emptyset\})=0$. But, with equation \eqref{eq:Palm_measure}, $\mathrm{P}_{0}(\{\emptyset\})=0$, so that $\mathrm{P}_{0}(A)=0$. This finishes the proof.
\end{proof}

The preceding lemma motivates the following definition.
\begin{definition}
	$\mathrm{P}_{0}$ is called \textbf{ergodic} if $\mathrm{P}_{0}(A)=0$ or $\mathrm{P}_{0}({}^{c}A)=0$ for any $\tau$-invariant event $A\in \mathcal{M}(G)$.
\end{definition}
This definition is \textit{improper} as $\mathrm{P}_{0}$ is not stationary with respect to translations. Yet, it is convenient as with Proposition \ref{prop:Palm_ergodicity}, we are now allowed to say that $\mathrm{P}$ is ergodic if and only if $\mathrm{P}_{0}$ is ergodic. Similarly, we define joint-ergodicity.
\begin{definition}
	$(\mathrm{P}_{0}, \nu)$ are called \textbf{jointly-ergodic} if either $\mathrm{P}_{0}\otimes\nu(A)=0$ or $\mathrm{P}_{0}\otimes\nu({}^{c}A)=0$ for any $(\tau, \theta)$-invariant event $A \in \mathcal{M}(G)\times \F(G)$.
\end{definition}

As before, this definition is improper since $(\mathrm{P}_{0}, \nu)$ is not $(\tau,\theta)$-stationary, but it is convenient in virtue of the following lemma.

\begin{lemma}\label{lem:joint_ergodicity}
	The two following statements are equivalent:
	\begin{enumerate}
		\item $(\mathrm{P}, \nu)$ are jointly-ergodic. 
		\item $(\mathrm{P}_{0}, \nu)$ are jointly-ergodic.
	\end{enumerate}
\end{lemma}

\begin{proof}
	Assume $(1)$ is satisfied and let $A$ be a $(\tau,\theta)$-invariant set with $\mathrm{P}\otimes\nu(A)=0$. Combining Campbell's formula \eqref{eq:Palm_measure} with the weak-stationarity of $\nu$ and the invariance of $A$ yields
	\begin{align*}
		\mathrm{P}_{0}\otimes\nu(A)&=\int_{\mathrm{F}}\int_{\mathrm{M}}\int_{G} \mathds{1}((\tau_{t}\xi,B)\in  A))\omega(t)\,\xi(\mathrm{d}t)\,\mathrm{P}_{0}(\mathrm{d}\xi)\, \nu(\mathrm{d}B)\\
		\text{(weak-stationarity of $\nu$)}&=\int_{\mathrm{F}}\int_{\mathrm{M}}\int_{G} \mathds{1}((\tau_{t}\xi,\theta_{t}B)\in  A))\omega(t)\,\xi(\mathrm{d}t)\,\mathrm{P}_{0}(\mathrm{d}\xi) \,\nu(\mathrm{d}B)\\
		\text{($(\tau, \theta)$-invariance of $A$)}&=\int_{\mathrm{F}}\int_{\mathrm{M}}\int_{G} \mathds{1}((\xi,B)\in  A))\omega(t)\, \xi(\mathrm{d}t)\,\mathrm{P}_{0}(\mathrm{d}\xi)\,\nu(\mathrm{d}B)\\
		&=0
	\end{align*}
	as $A$ has null measure. Hence, $(2)$ is satisfied.\medbreak
	
	The converse derives from Campbell's inverse formula in a similar fashion.
\end{proof}

We now prove Theorem \ref{thm:Palm_perturabtion_ergodicity}.

\begin{proof}[Proof of Theorem \ref{thm:Palm_perturabtion_ergodicity}]
	Following the conclusion of Theorem \ref{prop:Palm_ergodicity}, it suffices to prove that either $\mathrm{P}_{0}^{\nu}(A)=0$ or $\mathrm{P}_{0}^{\nu}({}^{c}A)=0$ for any $\tau$-invariant event $A \in \mathcal{M}(G)$. By definition,
	\begin{equation*}
		\mathrm{P}_{0}^{\nu}(A) = \mathrm{P}_{0}\otimes\nu(\Gamma^{-1}(A)).
	\end{equation*}
	
	We will now prove that $\Gamma^{-1}(A)$ is $(\tau, \theta)$-invariant. This is a simple consequence of Lemma \ref{lem:technical_lemma_Palm_perturbation}. Indeed, if $\Gamma(\xi,B)\in A$, then $\Gamma(\tau_{t}\xi,\theta_{t}B)=\tau_{t+B_{t}}\Gamma(\xi,B)\in A$. \medbreak
	
	As $\mathrm{P}$ and $\nu$ are jointly-ergodic, it follows from Lemma \ref{lem:joint_ergodicity} that either
	\begin{equation*}
		\mathrm{P}_{0}\otimes\nu(\Gamma^{-1}(A))=0 \quad \text{or}\quad \mathrm{P}_{0}\otimes\nu(\Gamma^{-1}({}^{c}A))=0.
	\end{equation*}
	But, this is means exactly either $\mathrm{P}_{0}^{\nu}(A)=0$ or $\mathrm{P}_{0}^{\nu}({}^{c}A)=0$. Hence, $\mathrm{P}_{0}^{\nu}$ is ergodic.
\end{proof}

We conclude this paragraph with a counter-example of Proposition \ref{prop:Palm_ergodicity} when condition $(3)$ in Proposition \ref{prop:Palm_ergodicity} is extended to the product $\sigma$-algebra $\mathcal{M}(G)\otimes \G$.

\begin{example}\label{ex:ergodicity}
	Let $\smash{\mathrm{P}_{0}=\delta_{\Z}}$ be the Palm measure of the stationary lattice $\Z+U$, which is an ergodic point process. Consider a symmetric set  $B\subset \R$ such that $\Vol{B}=\Vol{{}^{c}B}=+\infty$ and let $A=\mathrm{M}(\R^{d})\times B$. As $B$ is symmetric, $A$ is $\rho$-invariant, and both $\mathrm{Q}(A)$ and $\mathrm{Q}({}^{c}A)$ are infinite. Hence, $\mathrm{Q}$ is not ergodic on $\mathrm{M}(\R)\otimes \R$.
\end{example}

\subsubsection{Perturbed Palm measures supported on $\mathrm{N}_{s}(G)$}\label{sec:counting_measure_proof}

This section is dedicated to the Proof of Corollary \ref{cor:atomic_measure}.

\begin{proof}[Proof of Corollary \ref{cor:atomic_measure}]
	Denote by $\xi^{\otimes_{2!}}$ the second factorial measure of $\xi \in \mathrm{N}_{s}(G)$, that is
	\begin{equation*}
		\xi^{\otimes_{2!}}(\varphi) =\int_{G^{2}} \mathds{1}(s\neq t) \varphi(s,t)\, \xi(\mathrm{d}s)\,\xi(\mathrm{d}t).
	\end{equation*}
	$\xi^{\otimes_{2!}}$ is measure on $G^{2}$ which counts the number of pairs of distinct points of $\xi$.\medbreak
	
	For $\xi \in \mathrm{N}_{s}(G)$ and $B \in \mathrm{F}(G)$, we introduce 
	\begin{equation*}
		N(\xi,B)=	\int_{G^{2}}\mathds{1}(s+B_{s}=t+B_{t})\,\xi^{\otimes_{2!}}(\mathrm{d}s, \mathrm{d}t).
	\end{equation*}
	$N$ counts the number of distinct atoms $s,t$ of $\xi$ which are sent to the same location after $\xi$ has been perturbed by $B$, \textit{i.e.} $s+B_{s}=t+B_{t}$. Consequently, $\mathrm{P}_{0}^{\nu}$ is supported on $\mathrm{N}_{s}(G)$ if and only if
	\begin{equation*}
		\int_{\mathrm{F}}\int_{\mathrm{N}_{s}} N(\xi,B)\,  \mathrm{P}_{0}(\xi)\, \nu(\mathrm{d}B) =0.
	\end{equation*}
	But,
	\begin{equation*}
		\begin{aligned}
			\int_{\mathrm{F}}\int_{\mathrm{N}_{s}} N(\xi,B)\,  \mathrm{P}_{0}(\xi)\, \nu(\mathrm{d}B)&= \int_{\mathrm{N}_{s}}\int_{\mathrm{F}}\int_{G^{2}}\mathds{1}(s-t=B_{t}-B_{s})\,\xi^{\otimes_{2!}}(\mathrm{d}s, \mathrm{d}t)\,\nu(\mathrm{d}B)\,\mathrm{P}_{0}(\mathrm{d}\xi)\\
			&=\int_{\mathrm{N}_{s}}\int_{\mathrm{F}}\int_{G^{2}}\mathds{1}(s-t=\theta_{s}B_{t-s})\,\xi^{\otimes_{2!}}(\mathrm{d}s, \mathrm{d}t)\,\nu(\mathrm{d}B)\,\mathrm{P}_{0}(\mathrm{d}\xi)\\
			\text{(weak-stationarity of $\nu$)}&=\int_{\mathrm{N}_{s}}\int_{\mathrm{F}}\int_{G^{2}}\mathds{1}(s-t=B_{t-s})\,\xi^{\otimes_{2!}}(\mathrm{d}s, \mathrm{d}t)\,\nu(\mathrm{d}B)\,\mathrm{P}_{0}(\mathrm{d}\xi)\\
			\text{($\nu$ pointwise non-atomic)}&=\int_{\mathrm{N}_{s}}\int_{G}\int_{G}\underbrace{\nu_{t-s}(\{s-t\})}_{=0 \text{ as $s\neq t$}}\,\xi^{\otimes_{2!}}(\mathrm{d}s, \mathrm{d}t)\,\mathrm{P}_{0}(\mathrm{d}\xi)\\
			&=0.
		\end{aligned}
	\end{equation*}
	As $\mathrm{P}_{0}^{\nu}$ is supported on $\mathrm{N}_{s}(G)$, so is $\mathrm{P}^{\nu}$. This finishes the proof.
\end{proof}

\subsubsection{On the finiteness of perturbed Palm measures}\label{sec:probabilistic_proof}

This section is devoted to the proofs of the results stated in Section \ref{sec:probabilistic_interpretration}. We start with Proposition \ref{prop:compact_finite_perturbed_palm_measure}.

\begin{proof}Proof of Proposition \ref{prop:compact_finite_perturbed_palm_measure}
	The proposition derives from Campbell's inverse formula \eqref{eq:inverse_Palm_measure} with
	\begin{equation*}
		k(\xi,t) = \frac{1}{\xi(G)}, \quad \xi \neq \emptyset.
	\end{equation*}
	
	As $G$ is compact, $k$ is well-defined for any locally finite measure $\xi$.\medbreak
	
	The proof relies on the following observation,
	\begin{equation*}
		\xi_{B}(G) = \int_{G}\mathds{1}_{G}(t+B_{t})\, \xi(\mathrm{d}t)=\xi(G).
	\end{equation*}
	Hence, 
	\begin{align*}
		k(\tau_{-t}\xi_{B},t)&=k(\xi_{B},t)\\
		&=k(\xi,t)\\
		&=k(\tau_{-t}\xi, t)
	\end{align*}
	so that
	\begin{align*}
		\mathrm{P}^{\nu}(\mathrm{M})&=\int_{\mathrm{F}}\int_{\mathrm{M}}\int_{G}k(\tau_{-t}\xi_{B},t) \,\mathrm{d}t\, \mathrm{P}_{0}(\mathrm{d}\xi)\,\mathrm{\nu}(\mathrm{d}f)\\
		&=\int_{\mathrm{F}}\int_{\mathrm{M}}\int_{G} k(\tau_{-t}\xi,t)\, \mathrm{d}t\, \mathrm{P}_{0}(\mathrm{d}\xi)\,\mathrm{\nu}(\mathrm{d}f)\\
		&=\nu(\mathrm{F})\int_{\mathrm{M}}\int_{G} k(\tau_{-t}\xi,t)\, \mathrm{d}t\, \mathrm{P}_{0}(\mathrm{d}\xi)\\
		&=\mathrm{P}(\mathrm{M})\nu(\mathrm{F}),
	\end{align*}
	which is finite if and only if both $\mathrm{P}$ and $\nu$ are finite.
\end{proof}

We follow now with the counter-example introduced in Proposition \ref{prop:counter_example_probability}.

\begin{proof}[Proof of Proposition \ref{prop:counter_example_probability}]
	Denote by $\esp{.}$ the expectation with respect to $B$. As $B$ has increasing sample paths, equation \eqref{eq:inverse_palm_real_line} yields 
	\begin{equation*}
		\mathrm{P}^{\nu}(\mathrm{M}) = \esp{1+B_{1}} =+\infty.
	\end{equation*}
\end{proof}

\section{Appendices}\label{app:compact-open}

This paragraph can be omitted at first reading as it deals mainly with technical details related to $\sigma$-algebras and measurability. The goal of this appendix is to prove that $\Gamma$, as defined in Section \ref{sec:perturbation_palm}, is measurable. To this end, we will add some topological structure on both $\mathrm{M}(G)$ and $\mathrm{F}(G)$.\medbreak

The \textbf{weak-* topology} is the topology on $\mathrm{M}(G)$ induced by the mappings 
\begin{equation*}
	\Lambda_{\varphi} : \left\{
	\begin{array}{cccc}
		&\mathrm{M}(G)&\longrightarrow&\R\\
		&\xi&\longrightarrow&\xi(\varphi)
	\end{array}
	\right..
\end{equation*}
where $\varphi$ ranges through $\mathcal{C}_{c}(G, \R)$ the set of compactly supported real-valued continuous function.\medbreak

As for $\mathrm{F}(G)$, recall it is subset of continuous functions. Hence, it can be endowed with the induced compact-open topology generated by the subbasis 
\begin{equation*}
	V_{K,U}=\{f \in \mathrm{F}(G) : f(K)\subset U\}
\end{equation*}
where $K$ is compact and $U$ open. This topology was initially introduced by Fox \cite{Fox1945} and is relevant as it extends the topology of uniform convergence on all compact sets to non-metric spaces. For more details and a formal introduction on the topic, see Bredon \cite[Section VII.2]{Bredon1993}. \medbreak

The weak-* and the open-compact topologies induce two Borel $\sigma$-fields denoted respectively by $\borel{\mathrm{M}}$ and $\borel{\mathrm{F}}$. The first step into establishing the measurability of $\Gamma$ consists into proving that the aforementioned $\sigma$-algebras coincide with the standard algebra on $\mathrm{M}(G)$ and $\mathrm{F}(G)$.

\begin{proposition}\label{prop:cylindrical_algebra}
	$\mathcal{M}(G)=\borel{\mathrm{M}}$ and $\mathcal{F}(G)=\borel{\mathrm{F}}$
\end{proposition}

\begin{proof}
	The first equality is a routine argument which is left to the reader as it follows almost immediately from the definition.\medbreak
	
	As for the second one, recall that $\mathcal{F}(G)$ is the cylindrical $\sigma$-algebra, that is the smallest $\sigma$-algebra for which the evaluation maps $\pi_{t} : f \in \mathrm{F}(G)\longmapsto f(t) \in G$ are all measurable. From \cite[Proposition 2.3]{Bredon1993}, $\pi_{t}$ is continuous with respect to the compact-open topology, hence it is measurable with respect to $\borel{\mathrm{F}}$. Consequently, $\mathcal{F}(G) \subset \borel{\mathrm{F}}$ \medbreak
	
	Conversely, as $G$ is locally compact and second countable, it is separable. If $K \subset G$ is compact and $\{t_{n} : n \in \N\}$ is a dense subset of $K$, then for any $U$ open, one has
	\begin{equation*}
		V_{K,U} = \bigcap_{n \in \N}\pi_{t_{n}}^{-1}(U)\in \mathcal{F}(G).
	\end{equation*}
	As the sets $V_{K,U}$ generates  $\borel{\mathrm{F}}$, it follows easily that $\borel{\mathrm{F}}\subset \mathcal{F}(G)$.
\end{proof}

To establish the measurability of $\Gamma$, it would thus suffice to prove that $\Gamma$ is continuous. This will be true when $G$ is compact and we start with this case.

\begin{proposition}\label{prop:measurability}
	If $G$ is compact, $\Gamma$ is continuous. 
\end{proposition}

\begin{proof}
	Denote by $\mathcal{C}(G,G)$ denotes the space of real-valued continuous functions on $G$ equipped with the compact-open topology.\medbreak
	
	By definition of the induced topology, $\Gamma$ is continuous if and only if
	\begin{equation*}
		\Lambda_{\varphi}\circ \Gamma(\xi,f)=\int_{G}\varphi(t+f(t)) \, \xi(\mathrm{d}t)=\Lambda_{\varphi\circ (\Id+f)}(\xi)
	\end{equation*}
	is continuous for any $\varphi \in \mathcal{C}(G,\R)$. As $f \in \mathrm{M}(G) \longrightarrow \varphi \circ (\Id+f) \in \mathcal{C}(G,\R)$ is continuous in virtue of \cite[Theorem 2.10]{Bredon1993}, this amounts to prove
	\begin{equation*}
		\Pi : \left\{
		\begin{array}{cccc}
			&\mathrm{M}(G)\times\mathcal{C}(G,\R)&\longrightarrow&\R\\
			&(\xi,f)&\longrightarrow&\xi(f)
		\end{array}
		\right.
	\end{equation*}
	is continuous.\medbreak
	
	Fix $(\xi_{0},f_{0}) \in \mathrm{M}(G)\times\mathcal{C}(G,\R)$ and $\varepsilon>0$. We will exhibit an open neighbourhood $W$ of $(\xi_{0},f_{0})$ so that
	\begin{equation*}
		\left|\Pi(\xi,f) -\Pi(\xi_{0},f_{0})\right|<\varepsilon, \quad (\xi,f) \in W.
	\end{equation*}
	This will be sufficient to establish the continuity of $\Pi$, and \textit{a fortiori} $\Gamma$.\medbreak
	
	Consider 
	\begin{equation*}
		U=\left\{\xi \in \mathrm{M}(G) : |\Lambda_{f_{0}}(\xi)-\Lambda_{f_{0}}(\xi_{0})| < \frac{\varepsilon}{2} \text{ and }  \xi(G)< 2\xi_{0}(G)\right\}.
	\end{equation*}
	By definition of the induced topology, $U$ is an open-neighbourhood of $\xi_{0}$.\medbreak
	
	On the other hand, as $G$ is compact, the compact-open topology coincides with the topology of uniform convergence, see \cite[Theorem 2.12]{Bredon1993}. Thus, 
	\begin{equation*}
		V=\left\{f \in \mathrm{C}(G,\R) : \|f-f_{0}\|_{\infty}<\frac{\varepsilon}{4\xi_{0}(G)}\right\},
	\end{equation*}
	is an open neighbourhood of $f_{0}$. Here $\|.\|_{\infty}$ denotes the norm of uniform convergence on $G$.\medbreak
	
	Set $W=U\times V$ and note that for any $(\xi,f)\in W$,
	\begin{align*}
		\left|\Pi(\xi,f) -\Pi(\xi_{0},f_{0})\right|&\leq \left|\Pi(\xi,f) -\Pi(\xi,f_{0})\right| +\left|\Pi(\xi,f_{0}) -\Pi(\xi_{0},f_{0})\right|\\
		&=\left|\int_{G}f(t)-f_{0}(t)\,  \xi(\mathrm{d}t)\right|+\left|\Lambda_{f_{0}}(\xi_{0})-\Lambda_{f_{0}}(\xi)\right|\\
		&<\xi(G) \|f-f_{0}\|_{\infty}  + \left|\Lambda_{f_{0}}(\xi_{0})-\Lambda_{f_{0}}(\xi)\right|\\
		&<\varepsilon.
	\end{align*}
	in virtue of the definition of $W$. This finishes the proof.
\end{proof}

When $G$ is no longer compact, we approximate $\Gamma$ as simple limit of continuous function.

\begin{proposition}\label{prop:prop:measurability}
	$\Gamma$ is measurable.
\end{proposition}

\begin{proof}[Proof of Proposition \ref{prop:measurability}]
	As $G$ is second-countable and locally compact, it is $\sigma$-compact. Let $K_{n}$ be a sequence of increasing compact sets such that
	\begin{equation*}
		\bigcup_{n\in \N}K_{n}=G,
	\end{equation*}
	and consider
	\begin{equation*}
		\Gamma_{n}(\xi, f)=\Gamma(\restr{\xi}{K_{n}},f).
	\end{equation*}
	With Proposition \ref{prop:measurability}, $\Gamma_{n}$ is continuous, hence measurable. On the other, $\Gamma_{n}$ converges simply to $\Gamma$ as for all $\varphi \in \mathrm{C}_{c}(G, \R)$, 
	\begin{equation*}
		\begin{aligned}
			\Lambda_{\varphi}\circ \Gamma_{n}(\xi, f) =& \int_{K_{n}} \varphi(t+f(t)) \, \xi(\mathrm{d}t)\\
			\text{(dominated convergence theorem)}\longrightarrow& \int_{K_{n}} \varphi(t+f(t)) \, \xi(\mathrm{d}t)\\
			=&\Lambda_{\varphi}\circ\Gamma(\xi, f).
		\end{aligned}
	\end{equation*}
	As limit of continuous functions, it follows that $\Gamma$ is measurable. This finishes the proof.
\end{proof}

\clearpage
\bibliographystyle{alpha}
\bibliography{Reference.bib}

\end{document}